\documentclass[11pt]{article}
    
\usepackage{amsmath,amsfonts,amsthm,amscd,amssymb,graphicx}
\usepackage{hyperref}

\usepackage[utf8]{inputenc}

\usepackage{subfigure}
\usepackage{xcolor}

\numberwithin{equation}{section}

\usepackage{hyperref}

\def\bel{\begin{equation*} \begin{aligned}}
\def\eel{\end{aligned} \end{equation*}}
\def\beln{\begin{equation} \begin{aligned}}
\def\eeln{\end{aligned} \end{equation}}
\def\beq{\begin{equation}}
\def\eeq{\end{equation}}
\newcommand{\bea}{\begin{eqnarray}}
\newcommand{\eea}{\end{eqnarray}}

\newtheorem{theorem}{Theorem}[section]

\newtheorem{lemma}[theorem]{Lemma}

\def\rit{{\mathbb{R}}}
\def\cit{{\mathbb{C}}}

\def\Ai{\operatorname{Ai}}

\def\Ti{\operatorname{Ti}}

\begin{document}

\centerline{\Large \bf The dispersion relation of Tollmien-Schlichting waves} 

\bigskip

\centerline{Dongfen Bian\footnote{School of Mathematics and Statistics, Beijing Institute of Technology, $100081$, Beijing, China. Emails:
biandongfen@bit.edu.cn, daishoui@outlook.com and emmanuelgrenier@bit.edu.cn}, Shouyi Dai\footnotemark[1],
Emmanuel Grenier\footnotemark[1] 
}


\subsubsection*{Abstract}


It is well-known that shear flows in a strip or in the half plane are unstable for the Navier-Stokes
equations if the viscosity $\nu$ is small enough, provided the horizontal wave number $\alpha$
lies in a small interval, between the so called lower and upper marginal stability curves.
The corresponding instabilities are called Tollmien-Schlichting waves.
In this letter, we give a simple presentation of the dispersion relation of these waves 
and study its mathematical properties.


\section{Introduction}


In this paper, we address the classical question of the stability of shear flows for the following
incompressible Navier-Stokes equations in the half space $\Omega = \rit \times \rit_+$,
 \beq \label{NS1} 
\partial_t u  + (u  \cdot \nabla) u  - \nu \Delta u  + \nabla p  = f ,
\eeq
\beq \label{NS2}
\nabla \cdot u  = 0 ,
\eeq
together with the Dirichlet boundary condition $u  = 0$ when $y = 0$.

In these equations, $u$ is the velocity of the fluid, $p$ its pressure and $f$ a given external force. 
A shear flow is a stationary solution of (\ref{NS1},\ref{NS2}) of the form
$$
U(y) =  \Bigl(U_s(y),0 \Bigr), \qquad 
f  = \Bigl( - \nu \Delta U_s(y),0 \Bigr),
$$
where $U_s(y)$ is a smooth function, vanishing at $y = 0$ and converging exponentially fast
at infinity to some non zero constant $U_+$. We also assume that $U_s'(y)$ and $U_s''(y)$ converge exponentially fast
to $0$ as $y$ goes to infinity.

We will also consider the case of the strip $\Omega = \mathbb{R} \times [-1,+1]$
with the corresponding boundary condition $u  = 0$ at $y = \pm 1$.
In this case we will restrict ourselves to even solutions.
A shear flow is then a smooth function of $y$ which is even and vanishes at $y = \pm 1$.

\medskip 

It is well known in physics that such shear flows are spectrally unstable (namely, the
corresponding linearized equations admit an exponentially growing solution) provided
the Reynolds number $\mathfrak{R} = \nu^{-1}$ is large enough and provided
the horizontal wave number of the perturbation lies in some interval $[\alpha_-(\nu),\alpha_+(\nu)]$
depending on the viscosity.
Such instabilities are called Tollmien-Schlichting waves.
We in particular refer to \cite{Reid,Reid2,Schmidt} or to \cite{BG1,Zhang,Guo} for detailed discussions
from the physical or mathematical points of view. 

\medskip

It turns out that the dispersion relation of  the Tollmien-Schlichting waves is extremely delicate to obtain. 
The first aim of this letter is to give a simple and educational way to obtain it, without proof
(which can be found in \cite{BG1,Zhang,Guo}).
The second aim is to describe in details the corresponding instabilities through various asymptotics 
which are usually not considered in physics (where  only the behaviors of $\alpha_-(\nu)$ and $\alpha_+(\nu)$ 
are usually
studied).


\section{The dispersion relation}



\subsection{The Orr-Sommerfeld and Rayleigh equations}


The linearized Navier-Stokes equations near the shear flow $U$ are
\beq \label{LNS}
\partial_t v + (U \cdot \nabla) v + (v \cdot \nabla) U - \nu \Delta v + \nabla q = 0,
\eeq
\beq \label{LNS2}
\nabla \cdot v = 0
\eeq
with the boundary condition $v = 0$ on $\partial \Omega$.
We introduce the stream function of $v$ and take its Fourier transform in the $x$ variable with dual variable $\alpha$. Looking for instabilities of the linearized Navier-Stokes equations is then equivalent to looking for solutions of the form 
$$
\zeta(x,y) = \nabla^\perp \Bigl[ e^{i \alpha (x - ct)}  \psi(y) \Bigr]
$$
where $c \in \mathbb{C}$ with $\Im c > 0$.
Taking the curl of (\ref{LNS})
and dividing by $i \alpha$, we obtain the classical Orr-Sommerfeld equation on $\psi$
\beq \label{OS}
(U_s - c)  (\partial_y^2 - \alpha^2) \psi - U_s''  \psi  
- { \nu \over i \alpha}   (\partial_y^2 - \alpha^2)^2 \psi = 0
\eeq
with, in the half space case,  the boundary conditions $\psi(0) = \partial_y \psi(0) = 0$
and $\psi(y) \to 0$ as $y$ goes to infinity.

When $\nu = 0$, the Orr-Sommerfeld equation degenerates into the Rayleigh equation
\beq \label{Rayleigh}
(U_s - c) (\partial_y^2 - \alpha^2) \psi = U_s'' \psi,
\eeq
together with the boundary conditions $\psi(0) = 0$ and $\psi(y) \to 0$ as $y \to + \infty$.


\subsection{The dispersion relation}


Let us first discuss the half space case. At infinity, classical results on ordinary differential equations indicate
that solutions $\psi$ of the Orr-Sommerfeld equation  behave like $\exp(\lambda y)$ where
$$
\Bigl[ U_+ - {\nu \over i \alpha} (\lambda^2 - \alpha^2) \Bigr] (\lambda^2 - \alpha^2) = 0,
$$
namely where $\lambda = \pm \lambda_s$ or $\lambda = \pm \lambda_f$ with
$$
\lambda_s = |\alpha| \in \rit , \qquad
\lambda_f = \sqrt{ \alpha^2 + {i \alpha \over \nu} U_+} \in \cit .
$$
As $U_+ \ne 0$, we note that $|\lambda_f | \gg 1$ when $\nu$ goes to $0$ provided $\alpha / \nu \to + \infty$, 
which will be the case for the Tollmien-Schlichting waves.
Thus the Orr-Sommerfeld equation has four independent solutions.
Two of them, called $\psi_{f,\pm}$, have a ``fast behavior" and behave like $\exp( \pm \lambda_f y)$ at infinity.
Two of them, called $\psi_{s,\pm}$, have a ``slow behavior" and behave like $\exp(\pm \lambda_s y)$ at infinity.
Moreover, two of them, called $\psi_{f,-}$ and $\psi_{s,-}$ converge to $0$ as $y \to + \infty$,
and two of them, called $\psi_{f,+}$ and $\psi_{s,+}$ diverge at infinity.

We look for a non zero solution of (\ref{OS}) which go to $0$ as infinity and satisfies the boundary condition at $y = 0$. Such a solution
is necessarily a linear combination of $\psi_{f,-}$ and $\psi_{s,-}$, and exists if only if the following
dispersion relation holds true
\beq \label{disper1}
{\partial_y \psi_{f,-}(0) \over \psi_{f,-}(0)}
= {\partial_y \psi_{s,-}(0) \over \psi_{s,-}(0)}.
\eeq
In the strip, we only consider even solutions, and there exist two independent even solutions,
one, called $\psi_f$, being fast, and the other one, called $\psi_s$, being slow.
The dispersion relation is then
\beq \label{disper2}
{\partial_y \psi_{f}(1) \over \psi_{f}(1)}
= {\partial_y \psi_{s}(1) \over \psi_{s}(1)}.
\eeq
We now describe the slow and fast solutions.


\subsection{Study of the fast solution}


Let us first discuss the half space case and let us describe $\psi_{f,-}$. For a fast solution, we expect that
the gradients are very large. In this case, it can be established that $\psi_{f,-}$ can be approximated by
a solution of
\beq \label{OSfast}
(U_s - c) \partial_y^2 \psi - {\nu \over i \alpha} \partial_y^4 \psi = 0,
\eeq
just keeping the highest order derivatives in (\ref{OS}). We note that $\theta = \partial_y^2 \psi$
satisfies the Airy type equation
\beq \label{OSfast2}
(U_s - c) \theta =  {\nu \over i \alpha} \partial_y^2 \theta.
\eeq
We further approximate $\theta$ by a solution of the classical linear Airy equation
\beq \label{OSfast3}
\partial_y^2 \theta = {i \alpha \over \nu} U_s'(y_c) (y - y_c ) \theta,
\eeq
where $y_c$ is the so-called critical layer, defined by $U_s(y_c) = c$.
The solution of (\ref{OSfast3}) which decays at infinity is then simply given by
$\Ai( \gamma (y - y_c))$ where $\Ai$ is the Airy function 
and where
$$
\gamma = i^{1/3} \tilde \gamma, \qquad
\tilde \gamma  = \Bigl(  {\alpha U_s'(y_c) \over \nu} \Bigr)^{1/3}.
$$
More precisely, it is possible to prove that
$$
\frac{\partial_y \psi_{f,-}(0)}{ \psi_{f,-}(0)}
= \gamma {\Ai(- \gamma y_c,1) \over \Ai(-\gamma y_c,2)}
\Bigl[ 1 + o(1) \Bigr],
$$
where $\Ai(\cdot,1)$ and $\Ai(\cdot,2)$ are the first and second primitives of the Airy function
which vanish at infinity.
We introduce the Tietjens function, which is defined by
$$
\Ti(y) = \frac{\Ai(\xi,2)}{\xi \Ai(\xi,1)},
\qquad\xi = - i^{1/3} y,
$$
which gives
\beq \label{quotient1}
\frac{\partial_y \psi_{f,-}(0)}{ \psi_{f,-}(0)}
= - {1 \over y_c \Ti(\tilde \gamma y_c)} 
\Bigl[ 1 + o(1) \Bigr].
\eeq
Using the expansions at infinity of $\Ai$ \cite{DLMF}, we obtain the following expansion at infinity
$$
\Ti(y) =  {e^{i \pi / 4} \over y^{3/2}} \Bigl( 1 + {5 \over 4} {e^{i \pi / 4} \over y^{3/2}} 
+ {151 \over 32} {e^{i \pi / 2} \over y^3} + \cdots \Bigr).
$$
The case of the strip is similar, except that $\psi_{f}$ is evaluated at $y = 1$, which gives
$$
\frac{\partial_y \psi_{f}(0)}{ \psi_{f}(0)} = - {1 \over (1 - y_c) \Ti (\tilde \gamma (1 - y_c)) }
\Bigl[ 1 + o(1) \Bigr].
$$


\subsection{Study of the Rayleigh equation}


For slow solutions, it can be proved that the viscous term $\nu \alpha^{-1} (\partial_y^2 -\alpha^2)^2 \psi_{s,-}$
may be neglected in (\ref{OS}) and that $\psi_{s,-}$ may be approximated by a solution $\psi_R$ of the Rayleigh equation.
In particular, 
$$
\frac{\partial_y \psi_{f,-}(0)}{ \psi_{f,-}(0)} = \frac{\partial_y \psi_R(0)}{ \psi_R(0)} \Bigl[ 1 + o(1) \Bigr].
$$
We now compute the right hand side term in the half space and in the strip.


\subsubsection{Half space}


\begin{lemma}
Let $\psi_R$ be a solution of the Rayleigh equation which goes to $0$ as $y$ goes to infinity.
We have, as $\alpha$ and $c$ go to $0$,
\beq \label{halfs}
{\partial_y \psi_R(0) \over \psi_R(0)}
=  - {U_s'(0) \over c}   - {\alpha \over c^2} (U_+ - c)^2 + {\alpha^2 \over c^2} (U_+ - c)^4 \Omega_0(0,c)
+ O \Bigl( {\alpha^3 \over c^2} \Bigr)
\eeq
where
\beq \label{defiOmega0z}
\Omega_0(y,c) = -{1 \over ( U_+ - c)^2} \int_y^{+ \infty} \Bigl[
{(U_s(z) - c)^2 \over (U_+ - c)^2} - {(U_+ - c)^2 \over (U_s(z) - c)^2} \Bigr] \, dz . 
\eeq
\end{lemma}

\begin{proof}
Following \cite{Reid}, we introduce
$$
\Omega(y) = {\psi_R(y) \over (U_s(y) - c) 
[ U_s'(y) \psi_R(y) - (U_s(y) - c){\partial_y \psi_R(y)} ] } 
$$
and note that, after some computations, $\Omega(y)$ satisfies the ordinary differential equation
\beq \label{equationOmega}
\Omega' = \alpha^2 Y \Omega^2 - Y^{-1}
\eeq
where $Y(y) = (U_s(y) - c)^2$.  
We note that the source term $Y^{-1}$ is singular when $U_s(y) = c$.
Moreover, in view of the asymptotic behavior of $\psi_R$,
$$
\lim_{y \to + \infty} \Omega(y) = {1 \over \alpha ( U_+ - c)^2} .
$$
This leads to the following expansion
$$
\Omega(y) = {1 \over \alpha (1 - c)^2} + \Omega_0(y)
+ O(\alpha),
$$
where $\Omega_0$ is defined in (\ref{defiOmega0z}).
Moreover,
$$
\Omega(0) = -  {\psi_R(0) \over c [ U'_s(0) \psi_R(0) + c \partial_z \psi_R(0) ]},
$$
which leads to (\ref{halfs}).
\end{proof}

As $\Im c \to 0$, we note that
\beq \label{imomega0}
\Im \Omega_0(0,c) \to - i \pi {U_c'' \over U_c'^3} (U_+ -c)^2,
\quad
\Re \Omega_0(0,c) \sim  - {1\over U_c'^2 y_c} + O(\log y_c).
\eeq


\subsubsection{Strip}


Similarly, we approximate $\partial_y \psi_s(\pm 1) / \psi_s(\pm 1)$ by $\partial_y \psi_R(\pm 1) / \psi_R(\pm)$
where $\psi_R$ is the even solution of Rayleigh equation \ref{Rayleigh}.

\begin{lemma}
Let $\psi_R$ be an even solution of the Rayleigh equation.
As $\alpha \to 0$ and $\nu \to 0$, we have
\begin{align*}
{\partial_y \psi_R(1) \over \psi_R(1)} &= - {U_s'(1) \over c} 
+ {\alpha^2 \over c^2} \int_0^1 (U_s(z) - c)^2 \, dz
- {\alpha^4 \over c^2} \int_0^1 {\omega_2^2(z) \over (U_s(z) - c)^2} \, dz
+ O \Bigl( {\alpha^6 \over c^2} \Bigr)
\end{align*}
where
$$
\omega_2(y) = \int_y^0 (U_s(z) - c)^2 \, dz .
$$
\end{lemma}

\begin{proof}
We introduce 
$$
\omega(y) = {1 \over \Omega(y)} = (U_s(y) - c) { U_s'(y) \psi(y) - (U_s(y) - c) \psi'(y) \over \psi(y)} .
$$
A direct computation shows that
\beq \label{equationomega}
\omega'(y) = - \alpha^2 Y(y) + {\omega^2(y) \over Y(y)},
\eeq
where $Y(y) = (U_s(y) - c)^2$.
Moreover, $\omega(0) = 0$ since we only consider even modes.
We look for an expansion of $\omega(y)$ of the form
$$
\omega(y) = \omega_2(y) \alpha^2 + \omega_4(y) \alpha^4 + O( \alpha^6) .
$$
This leads to
$$
\omega_2'(y) = Y(y), \qquad
\omega_4'(y) = {\omega_2^2(y) \over Y(y)}.
$$
The solutions which vanish at $y = 0$ are
$$
\omega_2(y) = \int_y^0 Y(z) \, dz,
\qquad 
\omega_4(y) = - \int_y^0 {\omega_2^2(z) \over Y(z) } dz.
$$
This leads to
\beq \label{expansionomega1}
\omega(1) = - \alpha^2  \int_0^1 (U_s(z) - c)^2 \, dz + \alpha^4  \int_0^1 {\omega_2^2(z) \over (U_s(z) - c)^2} \, dz + O(\alpha^6).
\eeq
We note that
$$
\omega(1) = - c U_s'(1) - c^2 {\partial_z \psi(1) \over \psi(1)},
$$
which leads to the dispersion relation.
Moreover, if $I(c)$ is an integral of the form
$$
I(c) = \int_{y_1}^0 {f(z) \over Y(z)} dz, 
$$
where $f$ is a smooth function,
then we note that
\begin{equation}\label{imic0}
\Im I(c) = - \pi {U_c'' \over (U_c')^3} f(y_c)  + {\pi \over (U_c')^2} f^{\prime}(y_c),
\quad 
\Re I(c) = - {f(y_c) \over U_1' c}  + O( \log c)
\end{equation}
as $c \to 0$.
\end{proof}


\section{The half space case}


Tollmien-Schlichting waves, which appear as $\nu$ goes to $0$, are characterized by a small
$\alpha$ and a small $c$, such that $\alpha / c$ is bounded. We do not prove this fact here
and just study the dispersion relation (\ref{disper1}) in the regime $\nu \to 0$, $c \to 0$,
$\alpha \to 0$ with $\alpha / c$ bounded.

Let $\Lambda$ be defined by
$$
y_c = {c \over U_s'(0)} \Lambda.
$$
We note that $\Lambda = 1 + O(c)$.
After rewritting, the dispersion relation (\ref{disper1}) becomes
\begin{equation} \label{approx1}
\Lambda \Ti(\tilde \gamma y_c) 
= 1 - {\alpha \over c} {(U_+ - c)^2 \over  U_s'(0)} + {\alpha^2 \over c} {(U_+ - c)^4 \over  U_s'(0)}
\Bigl[ {1 \over U_s'(0) c} + \Omega_0(0,c) \Bigr] + O \Bigl( { \alpha^3 \over c} \Bigr) .
\end{equation}
We now study this dispersion relation.

\begin{lemma}
Provided $\nu$ is small enough, (\ref{approx1}) has one solution $c$ with $\Im c > 0$ provided $\alpha_-(\nu) < \alpha < \alpha_+(\nu)$,
where, as $\nu \to 0$,
$$
\alpha_-(\nu) \sim C_-  {(U_s')^{5/4} \over U_+^{3/2}} \nu^{1/4},
\qquad
\alpha_+(\nu) \sim \Bigl({1 \over 2 \pi^2} {U_s'(0)^{11} \over U_1''(0)^2}\Bigr)^{1/6} \nu^{1/6},
$$
with $C_- \approx 1.005$. We also have
$| \alpha \Im c | \lesssim \nu^{1/2}$.
Moreover, when $\alpha$ is of order $\nu^{1/6}$, namely when $\alpha = \nu^{1/6} \tilde \alpha$
with $\tilde \alpha$ of order one,
we have,
$$
c = {U_+^2 \over U_s'(0)} \alpha \Bigl[ 1 + \nu^{1/6} \tilde \sigma \Bigr]
$$
where
$$
\tilde \sigma = {1 \over \tilde \alpha^2} {(U_s')^{5/2} \over U_+^3} e^{i \pi /4} 
- i \pi \tilde \alpha {U_s'' U_+^2 \over (U_s')^3}
+ O(\nu^{1/6}).
$$
In particular, near the upper marginal stability curve, $\Im c = O(\nu^{1/3})$.
%
\end{lemma}

\begin{proof}
We first discuss the marginal stability curves, defined by $\Im c = 0$.
In this case, $y_c$ is real, and thus $z = \tilde \gamma y_c$ is real. We have
$$
\Re \Omega_0(0,c) = - {1 \over U_s'(0) c} + O (\log c).
$$
This leads to
\beq \label{reTi}
\Re \Ti(z) = 1 - {U_+^2 \over U_s'(0)} {\alpha \over c}  + O \Bigl( { \alpha^2 \log c \over c} \Bigr) + O(\gamma^{-2})
\eeq
and, using (\ref{imomega0}),
\beq \label{imTi}
\Im \Ti(z) \sim - \pi {\alpha^2 \over  c} {U_s''(0) U_+^4 \over U_s'^4(0)} .
\eeq
As the Tietjens function is bounded on the real line, this implies that $\alpha c^{-1}$ is bounded. Hence, as $\alpha \to 0$, $\Im \Ti(z) \to 0$.
Thus, either $z \to + \infty$ or $z \to z_0$, the only real zero of the Tietjens function
of the real line, which approximately equals $2.297$.

Let us study the first possibility. In this case, as $z \to + \infty$, $\Re \Ti(z) \to 0$.
We thus define $\sigma$ by
$$
c \sim { U_+^2 \over  U_s'(0)} \alpha (1 + \sigma).
$$
Then (\ref{approx1}) becomes
\begin{equation} \label{approx2}
\Ti(\tilde \gamma y_c) = 1 - {1 \over 1 + \sigma} - {i \pi \alpha \over 1 + \sigma}
{U_s'' U_+^2 \over (U_s')^3}.
\end{equation}
Using the expansion of $\Ti$ at infinity, we get
\begin{equation} \label{approx3}
- {\nu^{1/2} \over \alpha^2} {(U_s')^{5/2} \over U_+^3} {e^{i \pi /4} \over (1 + \sigma)^{3/2}}
= 1 - {1 \over 1 + \sigma} - {i \pi \alpha \over 1 + \sigma}
{U_s'' U_+^2 \over (U_s')^3}.
\end{equation}
In the regime $|\sigma| \ll 1$, this gives
\begin{equation} \label{approx4}
\sigma = {\nu^{1/2} \over \alpha^2} {(U_s')^{5/2} \over U_+^3} e^{i \pi /4} 
- i \pi \alpha {U_s'' U_+^2 \over (U_s')^3}.
\end{equation}
We note that $\sigma$ is real when 
$$
{\nu^{1/2} \over \alpha^2} {(U_s')^{5/2} \over U_+^3} 
= \sqrt{2} \pi \alpha {U_s'' U_+^2 \over (U_s')^3},
$$
namely when
\beq \label{branch2}
\alpha^6 = {1 \over 2 \pi^2} {U_s'(0)^{11} \over U_1''(0)^2} \nu .
\eeq
In this case, we are on the ``upper marginal stability curve".

In view of (\ref{branch2}), we rescale $\alpha$ and $\sigma$ according to
$$
\alpha = \nu^{1/6} \tilde \alpha, \qquad \sigma = \nu^{1/6} \tilde \sigma.
$$
Then (\ref{approx4}) gives
\begin{equation} \label{approx5}
\tilde \sigma = {1 \over \tilde \alpha^2} {(U_s')^{5/2} \over U_+^3} e^{i \pi /4} 
- i \pi \tilde \alpha {U_s'' U_+^2 \over (U_s')^3}.
\end{equation}
We note that $\Re c$ is of order $\nu^{1/6}$, whereas $\Im c$ is of order $\nu^{1/3}$.
Moreover, $y_c$ is also of order $\nu^{1/6}$, whereas $\gamma$ is of order $\nu^{-5 /18}$:
The critical layer is at a distance of order $\nu^{1/6}$ from the boundary. 
Its size, of order $\gamma^{-1} \sim \nu^{5/18}$, is much smaller
than its distance from the boundary: the critical layer is ``detached" from the boundary.

Let us now turn to the second possibility, namely to the case when $z \to z_0$. 
We have $\Ti(z_0) \approx 0.5644$.
Thus, on the marginal curve, we have
$$
0.5644 \approx 1 - {U_+^2 \over U_s'} {\alpha \over c}.
$$
In this case,
$$
c \approx 2.2959 {U_+^2 \over U_s'(0)} \alpha.
$$
As $\tilde \gamma y_c = z_0$, this gives
$$
{\alpha^{1/3} \over \nu^{1/3}} {c \over  (U_s')^{2/3}} = z_0,
$$
namely
$$
\alpha \approx 1.005 \nu^{1/4} {(U_s')^{5/4} \over U_+^{3/2}} . 
$$
In this case, $\alpha$, $c$ and $\gamma^{-1}$ are all of order $\nu^{1/4}$.
The critical layer is at a distance $\nu^{1/4}$ from the boundary, and its size is of the same magnitude. 
This case is referred to as the ``lower marginal stability" regime.

In this regime, the third term of the right hand side of (\ref{approx1}) can
be neglected and the dispersion relation can further be simplified into
\begin{equation} \label{approx10}
\Ti\left({\tilde \alpha^{1/3} \tilde c \over {U_s'}^{2/3}}\right)
= 1 - {\tilde \alpha \over \tilde c} {U_+^2 \over  U_s'}
\end{equation}
where $\tilde \alpha = \alpha \nu^{-1/4}$ and $\tilde c = c \nu^{-1/4}$.  When $\tilde \alpha \to + \infty$, we observe that $\tilde c \to + \infty$. We can then approximate the
Tietjens function $\Ti$ by its asymptotic expansion, which leads to
$$
 {e^{i \pi/4} U_s'\over (\tilde \alpha^{1/3} \tilde c)^{3/2}} \sim 1 
 - {\tilde \alpha \over \tilde c} {U_+^2 \over  U_s'}.
$$
We introduce $\sigma$ such that
$\tilde \alpha = \tilde c U_s' U_+^{-2}  (1 - \sigma)$,
which gives
$$
\sigma \sim {e^{i \pi /4} {U_s'}^{5/2}\over \tilde \alpha^2 U_+^3} (1 - \sigma)^{3/2},$$
thus
$$
\sigma \sim {e^{i \pi /4} {U_s'}^{5/2}\over \tilde \alpha^2 U_+^3}.
$$
This relation is of course valid as long as $\alpha \ll \nu^{1/6}$ where the third term of (\ref{approx1}) can no longer be neglected.
\end{proof}


\section{The strip case}


We repeat the previous analysis in the case of the strip.
After rewritting, we obtain
\begin{align} \label{approxs12}
\Lambda    \Ti&\Bigl( \tilde \gamma (1-y_c) \Bigr) 
= -1 -{\alpha^2 \over U_s'(1) c} \int_0^1 (U_s(z) - c)^2 \, dz +O \Bigl( {\alpha^6 \over c} \Bigr)\\
    & + \left( {\alpha^4 \over U_s^{\prime}(1) c} \int_0^1 {\omega_2^2(z) \over (U_s(z) - c)^2} \, dz
    - \frac{\alpha^4}{c^2 U_s^{\prime}(1)^2}\left(\int_0^1 (U_s(z) - c)^2 d\, z\right)^2 \right) .
\end{align}
As previously, we study the regime $\nu \to 0$, $\alpha \to 0$ and $c \to 0$, with $\alpha / c$ bounded.

\begin{lemma}
Provided $\nu$ is small enough, (\ref{approxs12}) has one solution $c$ with $\Im c > 0$ provided
$$
\alpha_-(\nu) < \alpha < \alpha_+(\nu)
$$
where, as $\nu \to 0$,
\begin{align*}
\alpha_-(\nu) & \sim  C_- \vert U_s' \vert^{5/7} 
\Bigl( \int_{0}^{+1} U_s^2(y) \, dy \Bigr)^{-3/7} \nu^{1/7}, \\
\alpha_+(\nu) & \sim (2\pi^2)^{-1/11} { \vert U_s' \vert \over ( U_1'')^{2/11}}
\Bigl( \int_{0}^{+1} U_s^2(y) \, dy \Bigr)^{-5/11} \nu^{1/11},
\end{align*}
with $C_- \approx 1.7302$.
Moreover, we have $| \alpha \Im c | \lesssim \nu^{5/11}$.
\end{lemma}

\begin{proof}
We first discuss the marginal stability curves, defined by $\Im c = 0$.
In this case, $1-y_c$ is real, and thus $z = \tilde \gamma (1-y_c)$ is real. We have
$$
\Re \left(\int_0^1 {\omega_2^2(z) \over (U_s(z) - c)^2} \, dz \right) = - {\omega_2^2(1) \over U_1' c}  + O( \log c).
$$
This leads to
\beq \label{reTistrip}
\Re \Ti(z) = -1 - {\alpha^2  \over U_1' c} \omega_2(1)   - O \Bigl( { \alpha^4 \log c \over c} \Bigr)
\eeq
and, using (\ref{imic0}),
\beq\label{imTistrip}
\Im \Ti(z) \sim  {\pi\alpha^4  U_c'' \over U_1'^4 c}  \omega_2^2(1).
\eeq
As the Tietjens function is bounded on the real line, this implies that $\alpha^2 c^{-1}$ is bounded. Hence, as $\alpha \to 0$, $\Im \Ti(z) \to 0$.
Thus, either $z \to + \infty$ or $z \to z_0$, the only real zero of the Tietjens function.

If $z \to + \infty$ then $\Re \Ti(z) \to 0$.
We thus define $\sigma$ by
$$
c \sim - { \alpha^2 \over U_1^\prime} \omega_2(1) (1 + \sigma).
$$
Then (\ref{approxs12}) becomes
\begin{equation} \label{approx5}
\Ti \Bigl(\tilde \gamma  (1- y_c) \Bigr) 
= - 1 + {1 \over 1 + \sigma} + {i \pi\alpha^2 \omega_2(1) U_c'' \over (1+\sigma) {U_1^\prime}^3 }.
\end{equation}

Using the expansion of $\Ti$, we get
\begin{equation} \label{approx6}
 {\nu^{1/2} \over \alpha^{7/2}} {(U_s')^{5/2} \over \omega_2(1)^{3/2}} {e^{i \pi /4} \over (1 + \sigma)^{3/2}}
= - 1 + {1 \over 1 + \sigma} + {i\pi\alpha^2 \omega_2(1)U_c''\over (1+\sigma) {U_1^\prime}^3 }.
\end{equation}
In the regime $|\sigma| \ll 1$, this gives
\begin{equation} \label{approx7}
\sigma = -{\nu^{1/2} \over \alpha^{7/2}} {(U_s')^{5/2} \over \omega_2(1)^{3/2}} e^{i \pi /4} 
+ {i\pi\alpha^2 \omega_2(1)U_c''\over  {U_1^\prime}^3 }.
\end{equation}
We note that $\sigma$ is real when 
\begin{equation}\label{equ:upperCurve}
{\nu^{1/2} \over \alpha^{7/2}} {(U_s')^{5/2} \over \omega_2(1)^{3/2}} 
= {\sqrt{2}\pi\alpha^2 \omega_2(1) U_c'' \over  {U_1^\prime}^3 } ,
\end{equation}
namely when
$\alpha^{11} \sim \omega_2(1)^{-5}\nu$.
In this case, we are on the ``upper marginal stability curve".
We rescale $\alpha$ and $\sigma$ according to
$\alpha = \nu^{1/11} \tilde \alpha$ and $\sigma = \nu^{2/11} \tilde \sigma$.
Then (\ref{approx7}) gives
\begin{equation}
\tilde \sigma = -\tilde\alpha^{-7/2} {(U_s')^{5/2} \over \omega_2(1)^{3/2}} e^{i \pi /4} 
+  {i\pi\tilde\alpha^2 \omega_2(1)U_c''\over  {U_1^\prime}^3 }.
\end{equation}
We note that $\Re c$ is of order $\nu^{2/11}$, whereas $\Im c$ is of order $\nu^{4/11}$.
Moreover, $1-y_c$ is also of order $\nu^{2/11}$, whereas $\gamma$ is of order $\nu^{-10 /33}$

Let us now turn to the second possibility, namely to the case when $z \to z_0$. 
We have $\Ti(z_0) \approx 0.5644$.
Thus, on the marginal curve, we have
$$
0.5644 \approx -1 - {\alpha^2  \over U_1' c} \omega_2(1).
$$
In this case,
$$
c \approx -0.6392 {\alpha^2 \over U_1'} \omega_2(1).
$$
As $\tilde \gamma (1 - y_c) = z_0$, this gives
$$
{\alpha^{1/3} \over \nu^{1/3}} {c \over  (U_s')^{2/3}} = z_0,
$$
which gives
$$
\alpha \approx 1.7302 \nu^{1/7} {\vert U_s' \vert^{5/7} \over \omega_2(1)^{3/7}}
$$
and ends the proof.
\end{proof}



\subsubsection*{Acknowledgments} D. Bian is supported by NSFC under the contract 12271032.

\subsubsection*{Conflict of interest}  The authors state that there is no conflict of interest.

\subsubsection*{Data availability}  Data are not involved in this research paper.


\end{document}